\documentclass[12pt, A4]{article}
\usepackage{authblk}

\usepackage{amssymb,comment}
\usepackage{amsmath,amsfonts,amsthm,amstext}
\usepackage[english]{babel}
\usepackage[latin1]{inputenc}
\usepackage{makeidx}
\usepackage{amscd}
\usepackage[dvips]{graphicx}
\usepackage{fancybox}
\usepackage{niceframe}
\usepackage{times}
\usepackage{mathrsfs}
\usepackage{bbm}
\usepackage[all]{xy}
\usepackage{fancyhdr}
\usepackage{sectsty}
\usepackage[Conny]{fncychap}
\usepackage[symbol]{footmisc}
\usepackage[T1]{fontenc}
\DefineFNsymbols{stars}[text]{* {**} {*} {\S} {\I}}
\setfnsymbol{stars}

\newtheorem*{thm*}{Theorem}

\newtheorem{thm}{Theorem}[section]
\newtheorem{cor}[thm]{Corollary}
\newtheorem{pro}[thm]{Proposition}
\newtheorem{deff}[thm]{Definition}
\newtheorem{lem}[thm]{Lemma}
\newtheorem{rem}[thm]{Remark}

\newcommand{\nc}{\newcommand}

\nc{\cc}{\D{C}} \nc{\hh}{\D{H}} \nc{\nn}{\D{N}} \nc{\oo}{\D{O}}
\nc{\qq}{\D{Q}}
 \nc{\rr}{\D{R}}
\nc{\zz}{\D{Z}} \nc{\livre}{\ast}

\nc{\barr}{\begin{array}} \nc{\earr}{\end{array}}
\nc{\bthm}{\begin{thm}} \nc{\ethm}{\end{thm}}
\nc{\bpro}{\begin{pro}} \nc{\epro}{\end{pro}}
\nc{\blem}{\begin{lem}} \nc{\elem}{\end{lem}}
\nc{\bins}{\begin{ins}} \nc{\eins}{\end{ins}}
\nc{\bcor}{\begin{cor}} \nc{\ecor}{\end{cor}}
\nc{\brem}{\begin{rem}} \nc{\erem}{\end{rem}}
\nc{\bdeff}{\begin{deff}} \nc{\edeff}{\end{deff}}
\nc{\bea}{\begin{eqnarray}} \nc{\eea}{\end{eqnarray}}
\nc{\D}[1]{{\mathbb#1}}
\def\R{\rm I\kern -.2em R}
\def\N{\rm I\kern -.18em N}
\def\Z{\rm Z\kern -.332em Z}
\def\de{\rm [\kern -.15em [}
\def\dd{\rm ]\kern -.15em ]}
\def\||{\hspace{0.15cm}|\hspace{0.15cm}}

\title{Subgroup Conjugacy Separability in Residually Free Groups }

\author[1]{S. C. Chagas\thanks{partially supported by Capes}}
\author[2]{I. Kazachkov\thanks{partially supported by Basque Government grant IT1483-22}}
\affil[1]{UnB - University of Brasília,Campus Universitário Darcy Ribeiro, Brasília-DF | CEP 70910-900, Brasil; \ \ email: \texttt{sheilachagas@gmail.com}}
\affil[2]{Ikerbasque - Basque Foundation for Science and Matematika Saila,  UPV/EHU,  Sarriena s/n, 48940, Leioa - Bizkaia, Spain; \ \  email:  \texttt{ilya.kazachkov@gmail.com}}
\date{}                     %% if you don't need date to appear
\begin{document}
\sloppy
\maketitle

\abstract{ 
	We prove that finitely presented residually free  groups are subgroup conjugacy separable. Furthermore, if they are of type $FP_\infty$, then  they are also subgroup conjugacy distinguished. 
	
	Using a connection between conjugacy separability and residual finiteness of outer automorphism group established by Grossman in \cite{Grossman}, we show that finitely presented residually free groups have residually finite outer automorphism groups.
}

\section{Introduction}

The three fundamental decision problems for groups: the word, conjugacy and the isomorphism problems, were formulated by Max Dehn in the early 1900's and since then have shaped entire areas of group theory.

One of the first important decidability results is due to A.~I.~Mal'cev, who observed a relation between decidability of the word problem and its residual properties, see \cite{M}. He proved that if a finitely presented group is residually finite, then it has decidable word problem.

In order to apply Mal'cev's argument and solve the conjugacy problem a stronger residual property is required. A group $G$ is said to be conjugacy separable if one can distinguish its conjugacy classes by looking at finite quotients of $G$. If a finitely presented group is conjugacy separable, then it has decidable conjugacy problem.

Another natural generalisation of residually finiteness that extends from elements to subgroups, is subgroup separability. A group $G$ is subgroup separable if one can distinguish its finitely generated subgroups by looking at its finite quotients. This property is more general and assures the decidability of the membership problem in $G$. 

Just as conjugacy separability generalises residual finiteness, the notion of subgroup conjugacy separability introduced by O.~Bogopolski and F.~Grunewald in \cite{BG} generalises the notion of conjugacy separability. A group $G$ is said to be subgroup conjugacy separable if for every pair of non-conjugate finitely generated subgroups $H$ and $K$ of $G$ there exists a finite quotient of $G$ where the images of these subgroups are not conjugate.  

Finitely presented subgroup conjugacy separable groups have solvable conjugacy problem for finitely generated subgroups. The latter means, that there is an algorithm which given a finitely presented subgroup conjugacy separable group $G = \langle X \mid R \rangle$ and two finite sets of elements $Y$ and $Z$ decides whether or not the subgroups $ \langle Y\rangle$ and $\langle Z\rangle$ are conjugate in $G$.

In this paper we study subgroup conjugacy separability in residually free groups. Residually free groups is a prominent class of groups that provides a context for a rich and powerful interplay among group theory, topology and logic. 

By definition, a residually free group embeds in a direct product of (perhaps, infinitely many) free groups. For example, the fundamental group of a closed orientable surface is residually free, but it cannot be embedded in a finite direct product if it has negative Euler characteristic. However, it is shown in \cite{BMR}, see also \cite{BHMS} that one can embed finitely generated residually free groups into a finite product of fully residually free groups. 

In turn, fully residually free groups have been extensively studied in connection with Tarski?s problems on the first order logic of free groups and they appear under different guises: as groups having the same existential theory as free groups, as coordinate groups of irreducible varieties as well as so-called limit groups.

A large body of work on the structure of residually free groups culminated in an important paper of Bridson-Howie-Miller-Short \cite{BHMS}. One of the outcomes of their work is a stark contrast between finitely generated residually free groups and residually free groups that satisfy higher finiteness properties. For instance, while the membership and conjugacy problem is undecidable for finitely generated residually free groups, they are decidble for finitely presented ones, see \cite{BHMS}.

The results of this paper constitute a furhter step in this direction and can be regarded as a generalisation of the aforementioned result from \cite{BHMS} as we prove

\begin{thm*}
	Let $G$ be a finitely presented residually free group. Then   $G$ is finitely presented subgroup conjugacy separable.
\end{thm*}

Moreover, we show that residually free group $G$ of type  $FP_\infty$  are also subgroup conjugacy distinguished, see Definition \ref{def:sub_conj_dist} and Theorem \ref{RFisCD} for the precise statement.

Note that if $G$ is residually finite and finitely presented, then for any finitely generated conjugacy distinguished subgroup $H$ of $G$ there exists an algorithm that decides if a given element $g$ of $G$ is conjugate to some element of $H$.

Finally, we establish residual finiteness of the outer automorphism group of a finitely presented residually free group, see Corollary \ref{cor:OutGResFinite}. In order to do so we use a result of Grossman, see \cite{Grossman}, who proved that if $G$ is a finitely generated conjugacy separable group such that all its pointwise inner automorphisms are inner, then $Out(G)$ is residually finite; we establish this for finitely generated residually free groups.

\section{Residually free groups of type $FP_\infty$} 

In this section we study conjugacy distinguished subgroups of residually free groups. We first prove that virtual retracts of hereditarily conjugacy separable groups are conjugacy distguished and then show that subgroups of type $FP_m$ (for large enough $m$) over $\mathbb{Q}$ of a finitely presented residually free group are conjugacy distinguished.

\medskip

We say that a group $G$ is hereditarily conjugacy separable if every finite index subgroup of $G$  is conjugacy separable.

Recall that a subgroup $H$ of a group $K$ is called a retract if there is a homomorphism $\rho : K \rightarrow H$ which restricted to $H$ is the identity map. This is equivalent to $K$ splitting as a semidirect product $N\rtimes H$, where $N = ker\,\rho$. In this case the map $\rho$ is called a retraction of $K$ onto $H$.

\begin{deff}\label{Def:VR} 
Let $G$ be a group and let $H$ be a subgroup of $G$. We  say that $H$ is a virtual retract of $G$ if there exists a subgroup $K\leq G$ such that $|G : K| < \infty$, $H \subseteq  K$ and $H$ is a retract of $K$.
\end{deff}

Given a group $G$, the profinite topology on $G$ has as a basis all the cosets of finite index subgroups of $G$. Each such coset is both open and closed in the profinite topology.
While the separability properties of $G$ can be described in terms of the finite quotients of $G$, it is sometimes more convenient to talk about the profinite topology instead. For instance, $G$ is called residually finite if for any $1\ne  g \in G$, there exists a finite quotient of $G$ in which $g$ does not map to the identity element. In other words, there is a finite index normal subgroup $N$ of $G$ such that $g \notin N$. Equivalently, this means that $\{1\}$ is closed in the profinite topology of $G$. In fact, this is the same as saying that any one element subset of $G$ is closed or, the seemingly stronger statement, that $G$ is Hausdorff.

The group $G$ is called subgroup separable if for every finitely generated subgroup $H$ of $G$ and every $g \in  G \setminus H$ there exists a normal subgroup of finite index, $N$ of $G$ such that $g \notin  HN$. Thus, there is a finite quotient of $G$ in which the images of $g$ and $H$ are disjoint. As before, this is the same as saying that every finitely generated subgroup is closed in the profinite topology of $G$.
More generally, a subset $S \subseteq  G$ is called separable if for every $g \notin S$, there is a finite quotient of $G$ in which $g$ and $S$ have disjoint images. Equivalently, $S$ is separable if it is closed.

A group $G$ is called conjugacy separable if the conjugacy classes of elements are separable.
If we have a subgroup $H$ of $G$, then there are two possible topologies one can put on $H$. Namely, the subspace topology and the profinite topology of $H$ itself. In general the subspace topology may be more coarse, but not if $H$ has finite index.

In the above notation, denote by $\widehat{ G}$ the profinite completion of $G$ and by $\overline{ H}$ the closure of $H$ in $\widehat{G}$.

\begin{deff} \label{def:sub_conj_dist}
	A group $G$ is called  subgroup conjugacy distinguished if for any finitely generated  subgroup $H$ of $G$ and any element $x\in G$ such that $x$ is not conjugate into $H$ there exists a finite quotient of $G$ where the image of $x$ is not conjugate into the image of $H$.
\end{deff}

A subgroup $H$ of a group $G$ is said to be conjugacy distinguished if $\cup_{g\in G} H^g$  is closed in the profinite topology of $G$. Equivalently, $H$ is conjugacy distinguished if whenever $y$ is not conjugate to an element of $H$, there exists a finite quotient of $G$ where the image of $y$ is not conjugate to an element of the image of $H$.

\begin{thm}\label{Thmconjdist} 
	Let $G$ be a  hereditarily conjugacy separable group,  and $H$  a virtual retract of $G$.  Then $H$ is conjugacy distinguished.
\end{thm}
\begin{proof} 
To prove the theorem, we shall replace $H$ and $g$ by their conjugates in $G$ and change $\gamma$ accordingly until we achieve the statement of the theorem holding for them.
	
Since  $H$ is a virtual retract of $G$, there exist a finite index subgroup $U$ of $G$ containing $H$  and a homomorphism $f : U\rightarrow H$ such that $f_{|_H}$ is the identity map.  Moreover, $G\overline{U}= \widehat{G}$.

Hence  replacing $g$ by some conjugate in $G$ we may assume that $\gamma$  belongs to $\widehat U$.  Since $g^{\gamma}\in \overline{H}$,  $\overline{H}\leq \overline{U}$  and $U$ is closed in the profinite topology of $G$ we obtain $g\in \widehat U\cap G= U$.

Let  $\widehat{f} : \widehat{U} \rightarrow  \overline{H}$ be the continuous extension $f$. We have, $g^{\gamma}\in \overline{ H}$, so  $g^{\gamma} = \widehat{f}(g^{\gamma}) =  f(g)^{\widehat{f}(\gamma)}\in  \widehat{f}(\overline{H}) = \overline{H}$.  Then, $g$ and $f(g)$ are conjugate in $\overline{U}$. Since $U$ is conjugacy separable  by hypothesis, $g$ and $f(g)$ are conjugate in $U$ and lemma is proved.
\end{proof}

\begin{thm}\label{RFisCD} 
	Let $G$ be a finitely presented  residually free group. Then  there exists $m$ such that  all subgroups of $G$ of type $FP_m$ over $\mathbb{Q}$  are conjugacy distinguished.
\end{thm}

\begin{proof} 
Let $G$ be a finitely presented residually free group.  By  \cite[Corollary 19]{BHMS}, $G$ embeds in a direct product of finitely many limit groups $L= \prod_{i=1}^n L_i$. Choose  minimal such $n$. Without loss of generality we may assume that $G$ is a subdirect product of the $L_i$'s.
 
Let $H$ be a $FP_n$ subgroup of $G$. By \cite[Lemma 7]{BW}, $H$ is a virtual retract of  $L$.  Then  $H$ also is a virtual retract of $G$. By  \cite[Theorem 3.5]{CZForum},  $G$ is conjugacy separable.   Hence,  by Theorem \ref{Thmconjdist},  $H$ is conjugacy distinguished in $G$.
\end{proof}

\section{Finitely presented residually free groups}

The goal of this section is to prove that finitely presented residually free groups are finitely presented subgroup conjugacy separable.

\medskip 

\begin{deff}\label{full} 
Let $G_1, \ldots, G_n$ be  groups. A subgroup $H \leq G_1\times \cdots \times G_n$ is called full if $H$ intersects each factor non-trivially, i.e. $H \cap G_i\ne 1$  for all $i \in \{1, 2, \ldots, n\}$. 
\end{deff}

Notice that if $H$ is full, in particular it is nontrivial and if $H$ intersects trivially one of the factors, say $H \cap G_1= 1$, then $H$ is isomorphic to a subgroup of $G_2\times \cdots \times G_n$.

\begin{pro}\label{teo1} 
Let $G_1, G_2$ be  finitely presented subgroups of the direct product $L = \prod_{i=1}^{n}L_i$, where $L_i$ is a limit group. If $G_1$ and $ G_2$ are  conjugate in $\widehat{L}$,  then  they are conjugate in $L$.
\end{pro}
\begin{proof}
Suppose that $G_1$ is not full in $L$, say $G_1\cap L_1=1$ and assume that $G_2$ is full. Consider $L'= L/L_1$ and note that it is a subgroup of $L$.  So $G_1\cong  G_1L_1/L_1\leq L'$ and since any finitely presented subgroup of $L'$ is separable, $\widehat{G}_1\cong \widehat{G_1L_1/L_1}\leq \widehat{ L}'$. Hence $\widehat{ L}_1\cap \widehat{ G}_1=1$ by hopfian property of a finitely generated profinite group. 

By assumption, there exists $\gamma\in \widehat{L}$ so that ${\widehat{ G}_1}^{\gamma}=\widehat{ G}_2$. It follows that 
$$
\widehat{ G_1L_1/L_1}^{\gamma}=\widehat{G_2L_1/L_1}\cong {\widehat{ G}}_2/{\widehat{ G}}_2\cap {\widehat{L}}_1
$$ 
are conjugate, and so isomorphic.  Therefore ${\widehat{G}}_1\cong {\widehat{ G}}_2/{\widehat{ G}}_2\cap {\widehat{ L}}_1$. Since  a finitely generated profinite group is hopfian, it follows that ${\widehat{L}}_1\cap {\widehat{ G}}_2= 1$.  In particular we have $L_1\cap G_2 =1$. Thus, $G_1$ is full if and only if so is $G_2$. Hence, factoring out  $L_i$ that intersect $G_1$ and $G_2$ trivially we may assume that $G_1$ and $G_2$ intersect all $L_i$
non-trivially.

Let $\pi_i$ be the natural projection from $L$ onto $L_i$. Observe that  the images $\pi_i(\widehat G_1)$ and $\pi_i(\widehat G_2)$ are conjugate in $\widehat{L_i}$, since $L_i$ is  subgroup conjugacy separable by \cite{CZ-16}; $\pi_i(G_1)$ and $\pi_i(G_2)$ are conjugate in $L_i$, i.e. $\pi_i(G_1)^{l_i}=\pi_i(G_2)$ for some $l_i\in L_i$.  Thus, conjugating $G_1$ by the element $(l_1,\dots, l_n)$ if necessary, we may assume that $\pi_i(G_1)=\pi_i(G_2)= N_i$ for all $i$.
 
Now observe that $\pi_i(G_j)$ is a finitely generated  subgroup of $L_i$. By the main result of \cite{Wilton},  $L_i$ has a virtual retract, so there exist a finite  index $K_i$ that contained $G_i$ and a homomorphism $\psi: K_i \rightarrow \pi_i(G_j)$ such that $\psi_{|_{\pi_i(G_j)}}$ is the identity map. Thus, substituting $L_i$ by $K_i$ we may assume that $L_i= K_i$, and then $N_i$ is a retract of $L_i$ for every $i$.
 
By  \cite[Theorem 8]{BW}, the finitely presented subgroups $G_1$ and $G_2$ contain some term of the lower central series $\gamma_m(L)$, then  $\gamma_m(L)\subseteq G_1\cap G_2$. Hence $\gamma_m(\widehat{L})\subseteq \widehat{G}_1\cap \widehat{G}_2$.
 
Now in  the quotient $\widehat{L}/\gamma_m(\widehat{L})$ the images of $\widehat{G}_1/\gamma_m(\widehat{L})$ and $\widehat{G}_2/\gamma_m(\widehat{L})$ are conjugate.   By Theorem 7 in Chapter 4  of \cite{Segal}, $G_1/\gamma_m(L)$ and $G_2/\gamma_m(L)$ are conjugate in $L/\gamma_m(L)$.   It follows that $G_1$ and $G_2$ are conjugate in $L$. 
\end{proof}

We denote by $N_L(G)$ the normalizer of $G$ in $L$, that is
$$
N_L(G)=\{l\in L\mid lGl^{-1}=G\}.
$$

\begin{lem}\label{lemaNdenso} 
	Let $L= \prod_{i=1}^nL_i$ be the direct product of limit groups and $G$ be a full finitely presented subgroup of $L$. Then the normalizer $N_{L}(G)$ is dense in $N_{\widehat{L}}(\widehat{G})$.
\end{lem}
\begin{proof}
	By  \cite[Theorem 2.5]{CZ-16}, $L_i$ is hereditarily subgroup conjugacy separable  and hence, by  \cite[Lemma 2.3]{BZ}, $N_{L_i}(\pi_i(G))$ is dense in $N_{\widehat{L}_i}(\pi_i(\widehat{G}))$. Therefore replacing $L_i$ by $N_{L_i}(\pi_i(G))$ we may assume that $\pi_i(G)$ is normal in $L_i$. Then,  by  \cite[Theorem 1]{BH}, either the index of  $\pi_i(G)$ in $L_i$ is finite or $L_i$ is abelian. Put $K=\prod_{i=1}^n\pi_i(G)$. By  \cite[Thereom 8]{BW}, $G\geq \gamma_m(K)$.  Since $\gamma_m(\overline{K})\leq \widehat{G} \leq \overline{N_{L}(G)}$, it suffices to show  that $N_{L/\gamma_m(K)}(G/\gamma_m(K))$   is dense in $N_{\overline{L/\gamma_m(K)}}( \overline{G/\gamma_m(K)})$. But this is the subject of \cite[Proposition 3.3]{RSZ}. The lemma is proved.
\end{proof}

\begin{thm}\label{teo2} 
	Let $G$ be a finitely presented residually free group. Then $G$ is finitely presented subgroup conjugacy separable.
\end{thm}

\begin{proof} 
By  \cite[Corollary 19]{BHMS}, $G$ embeds in a direct product of finitely many limit groups $L= \prod_{i=1}^n L_i$. Choose minimal such $n$. As we argued in the proof of Proposition \ref{teo1}, without loss of generality we may assume that $G$ is a subdirect product of the $L_i$'s.
	
Let $G_1, G_2$ be finitely presented  subgroups of $G$ such that $\overline{G}_1^{\gamma} = \overline{G}_2$, for some $\gamma\in \widehat{G}$. By Theorem \ref{teo1},  $G_1$ and $G_2$ are conjugate in $L$, i.e. there exist $l\in L$ such that $G_1^l = G_2$. Hence, $\delta=\gamma l^{-1}\in N_{\widehat{G}}(\widehat{G}_1)$.  Then, $l=\delta^{-1}\gamma\in N_{\widehat{G}}(\widehat{G}_1)\widehat{G}$.
	
Now by \cite[Theorem 8]{BW},  $G\geq \gamma_m(L)$ for some $m$. Hence 	$N_{G}(G_1){G}\geq \gamma_m(L)$. Observe that $N_{G}(G_1){G}/\gamma_m(L)$ is closed in the profinite topology of $L/\gamma_m(L)$ (see \cite{LW}). Therefore $N_{G}(G_1){G}$ which is the preimage of $N_{G}(G_1){G}/\gamma_m(L)$ in $L$ is closed in the profinite topology of $L$.
	
By Lemma \ref{lemaNdenso},  $N_{G}(G_1)$  is dense in $N_{\widehat{G}}(\widehat{G}_1)$.  Therefore $\overline{N_{G}(G_1)G} =N_{\widehat{G}}(\widehat{G}_1)\widehat{G}$.  Thus $N_{\widehat{G}}(\widehat{G}_1)\widehat{G}\cap L = N_{G}(G_1){G}$,  so $l\in N_{\widehat{G}}(\widehat{G}_1)\widehat{G}\cap L $ and also $l\in N_{G}(G_1){G}$. Then we may write $l= ng_0$, where $n\in N_{G}(G_1)$ and $g_0\in G$. Therefore $G_1^{g_0}=G_2$ as desired.
\end{proof}

\begin{cor}[cf. \cite{BW}]
	Finitely presented residually free groups are conjugacy separable.
\end{cor}
\begin{proof}
Let $G$ be a finitely presented residually free group.  From the definition of being residually free, it follows that any 2-generated subgroup of $G$ is either free or free abelian. Hence, in particular $G$ does not contain non-abelian Baumslag-Solitar groups. 

It follows that  $\langle x \rangle$ and $\langle y\rangle$ are conjugate if and only if either so are $x$ and $y$ or $x$ and $y^{-1}$ and the statement follows.
\end{proof}

\section{Outer automorphisms of residually free groups} \label{sec:4}

In this section we first show that every commensurating endomorphism of a finitely generated residually free group is an inner automorphism. We combine this fact with a result of Grossman's to conclude that the outer automorphism group of a finitely presented residually free group is residually finite. Our arguments are adjustments of \cite[Section 9 and Theorem 1.6]{AMS}.

\medskip

Let  $H$ be a  subgroup of a group $G$ and let $\varphi \colon H \to G$ be a homomorphism. Then $\varphi$ is {\it commensurating} if for all  $h\in H$ there are $z\in G$ and $n,m\in \mathbb{Z}\setminus \{0\}$ such that $h^n=z\varphi(h)^m z^{-1}$.

\begin{comment}
\begin{lem}\label{lem:norm-centr} Let $H$ be a residually free group and let $H<G$, where $G$ is a direct product of limit groups. Suppose further that $H$ is a subdirect product of $G$.
	\begin{itemize}
		\item[(i)] If $N \lhd H$ is a normal subgroup which does not contain non-abelian free subgroups, then $N$ is central in $H$.
		\item[(ii)] The quotient of $H$ by its center $Z=Z(H)$ is centerless.
	\end{itemize}
\end{lem}
\begin{proof} To prove (i), suppose that $N$ is not central in $H$. Then there exist $h \in H\setminus \{1\}$ and $g \in N\setminus \{1\}$ such that $hg \neq gh$.
	Since any $2$-generated subgroup of a residually free group is either free or free abelian, the latter implies that $h$ and $g$ generate a free subgroup $F$, of rank $2$, in $G$. Since  $g \in F \cap N$, this intersection is a non-trivial normal subgroup of $F$, hence it is a non-abelian free group.
	This contradicts the assumption that $N$ has no non-abelian free subgroups. Therefore $N$ must be central in $H$.
	
	To verify (ii), let $N\lhd H$ be the full preimage of the center of $H/Z$ under the homomorphism $H \to H/Z$. Then  $N$ is nilpotent of class at most $2$, hence it satisfies the assumptions of (i), and therefore it must be central in $H$.
	Thus $N \leqslant Z$; on the other hand $Z \leqslant N$ by the definition of $N$. It follows that $N=Z$, and so the image of $N$ in $H/Z$ (i.e., the center of $H/Z$) is trivial.
\end{proof}
\end{comment}

We record the following observation.
\begin{rem}\label{rem:comm_end_of_ab} If $H$ is a free abelian group then the only commensurating endomorphisms of $H$ are endomorphisms of the form $h \mapsto h^s$ for some $s \in \mathbb{Z} \setminus \{0\}$ and for all $h \in H$.
\end{rem}

\begin{thm}\label{thm:comm_end-raag}
	Let $H$ be a finitely generated non-abelian residually free group. Then every commensurating  endomorphism $\varphi\colon H \to H$ is an inner automorphism of $H$.
\end{thm}
\begin{proof}
	Let $H$ be a residually free group and let $H<G$, where $G$ is the direct product of limit groups. Since subgroups of limit groups are again limit groups, without loss of generality, we further assume that $H$ is a subdirect product of $G$, see \cite{BHMS}.
	
	Let $G=L_0\times L_1 \times \dots \times L_l$ be the standard factorization of $G$, where $L_0$ is the center of $G$ and $L_1,\dots,L_l$ are all non-abelian limit groups, see \cite{BMR, BHMS}.
	Observe that $L_0$ is a finitely generated free abelian group and $l\ge 1$ as $H$ is non-abelian.
	Let $\pi_i\colon G \to L_i$ denote the canonical retraction, $i=0,1,\dots,l$. We have $\pi_i(H)=L_i$.
	
	Observe that $N_i:=H \cap L_i\lhd H$ is non-trivial whenever $i =1,\dots,l$, because otherwise $H$ would embed into the direct product of $L_0 \times L_1 \times L_{i-1} \times L_{i+1} \times \dots \times L_l$, see \cite{BHMS}.
	Moreover, $N_i$ is non-abelian and hence there is an element $h_i \in N_i\setminus\{1\}$ such that $E_{L_i}(h_i)=\langle {h_i}\rangle \subseteq N_i$, $i=1,\dots,l$, where 
	$$
	E_G(h)=\left\{x \in G \mid x h^k x^{-1}=h^l \text{ for some }k,l \in \mathbb{Z}\setminus\{0\}\right\}.
	$$

Suppose that $x h^k x^{-1}=h^l$ in a residually free group $H$. Let $\psi$ be a homomorphism to a free group $F$ so that $\psi(x h^k x^{-1})=\psi(h^l)\neq 1$ (and so $\psi(h)\neq 1$). In a free group $y g^k y^{-1}=g^l$ implies $k=l$, so the same is true in any residually free group. It follows that if $H$ is residually free, then $E_H(h)=C_H(h)$.

Now consider any commensurating endomorphism $\varphi\colon H \to H$. For each $i \in \{0,1,\dots,l\}$ let $B_i\lhd G$ denote the product of all $L_j$, $j \neq i$; thus $G=L_i B_i \cong L_i \times B_i$ and $B_i= \ker \pi_i$. By the hypothesis, for any $g \in H \cap B_i$, $\varphi(g) \in H$ and $\varphi(g)^m=u g^n u^{-1} \in B_i$ for some $m,n \in \mathbb{Z}\setminus \{0\}$ and $u \in G$.
	And since $G/B_i \cong L_i$ is torsion-free, we can conclude that $\varphi(g) \in B_i$. The latter shows that $\varphi$ preserves the kernel of the restriction of $\pi_i$ to $H$, $i=0,1,\dots,l$. Therefore $\varphi$ naturally induces an endomorphism $\varphi_i\colon \pi_i(H) \to \pi_i(H)$ for $i=0,1,\dots,l$, defined by the formula $\varphi_i(\pi_i(g))=\pi_i(\varphi(g))$ for all $g \in H$.
	
	Evidently, $\varphi_i$ is a commensurating endomorphism of $\pi_i(H)$ for each $i=0,1,\dots,l$. Therefore, by Remark \ref{rem:comm_end_of_ab}, there exists $s \in \mathbb{Z} \setminus\{0\}$ such that $\varphi_0(a)=a^s$ for all $a \in \pi_0(H)$. On the other hand, if $i \neq 0$, then $L_i$ is a non-abelian limit group and $\pi_i(H)=L_i$. Therefore there exists $w_i \in L_i$ such that $\varphi_i(a)=w_i a w_i^{-1}$ for all $a \in \pi_i(H)$ (here we used the fact that $E_{L_i}(\pi_i(H))=\cap_{h\in \pi_i(H)}E_{L_i}(h)=\cap_{h\in \pi_i(H)} C_{L_i}(h) =\{1\}$), $i=1,\dots,l$.
	
	Let $\psi \in Inn(G)$ be the inner automorphism defined by $\psi(g)=wgw^{-1}$ for all $g \in G$, where $w=w_1 \cdots w_l \in G$. Let us show that the endomorphism $\varphi$ is actually the restriction of $\psi$ to $H$. The preceding paragraph implies that this is true if the abelian factor $L_0$ is trivial, because in this case for every $g \in H$ one would have $g=\pi_1(g) \cdots \pi_l(g)$, and so
	$$
	\begin{array}{rl}
		\varphi(g)=\pi_1(\varphi(g)) \cdots \pi_l(\varphi(g))= &
		\varphi_1 \left(\pi_1(g)\right) \cdots \varphi_l \left( \pi_l(g)\right)=\\
		=& \pi_1(g)^{w_1} \cdots \pi_l(g)^{w_l}=g^w.
	\end{array}
	$$
	
	On the other hand, if $L_0$ is non-trivial, then $N_0=H \cap L_0$ is also non-trivial (by the minimality of $G$). So, pick any $h_0 \in N_0\setminus\{1\}$. Let $h_1 \in N_1= H\cap L_1$ be the element constructed above. Since $\varphi$ is commensurating and $h_0h_1 \in H$, there must exist $m,n \in \mathbb{Z}\setminus\{0\}$ and $u \in H$ such that
	$$
	\varphi(h_0h_1)^m=u (h_0h_1)^n u^{-1}=h_0^n u_1 h_1^n u_1^{-1}, \text{ where } u_1=\pi_1(u) \in L_1.
	$$
	But we also have $\varphi(h_0 h_1)=\varphi_0(h_0) \varphi_1(h_1) = h_0^s w_1h_1w_1^{-1}$. Therefore
	\begin{equation*}\label{eq:left-right}
		h_0^{sm} w_1h_1^mw_1^{-1}=h_0^n u_1 h_1^n u_1^{-1}.
	\end{equation*}

	Applying $\pi_0$ and $\pi_1$ to the above equation we obtain $h_0^{sm}=h_0^n$ and $u_1^{-1}w_1 h_1^m w_1^{-1} u_1=h^n$.
	The former yields that $n=sm$; and the latter shows that $u_1^{-1}w_1 \in E_{L_1}(h_1)=\langle{h_1}\rangle$,
	in particular this element commutes with $h_1$. Thus $h_1^m=h_1^n$, and so $m=n$. Consequently, $s=1$, which implies that $\varphi(g)=w g w^{-1}=\psi(g)$ for all $g \in H$.
	If $w \in H$ then the proof would have been finished. However, this may not be the case, so one more step is needed.
	
	Let $h_i \in N_i =H \cap L_i$, $i=1,\dots,l$, be the elements constructed above so that  $E_{L_i}(h_i)=\langle{h_i}\rangle \subseteq H$,
	and set $h= h_1 \cdots h_l \in H$. By the  assumption, there exist $m,n \in \mathbb{Z}\setminus\{0\}$ and $u \in H$ such that $\varphi(h)^m=u h^n u^{-1}$. On the other hand, we know that $\varphi(h)=w h w^{-1}$.
	Combining these two equalities one gets $ w h^m w^{-1}=u h^n u^{-1}$ in $G$. Applying $\pi_i$ yields that $u_i^{-1}w_i \in E_{L_i}(h_i)=\langle{h_i}\rangle$, where $u_i=\pi_i(u) \in L_i$, for $i=1,\dots,l$. It follows that
	for every $i=1,\dots,l$, there exists $t_i \in \mathbb{Z}$ such that $w_i=u_i h_i^{t_i}$ in $L_i$. Thus, denoting $u_0=\pi_0(u) \in L_0$, we achieve
	$$
	w=w_1 \cdots w_l= u_1 h_1^{t_1} \cdots u_l h_l^{t_l}=  u_0^{-1} u h_1^{t_1} \cdots h_l^{t_l}=u_0^{-1} v,
	$$
	where the element  $v=u h_1^{t_1} \cdots h_l^{t_l}$ belongs to $H$ by construction. Since  $u_0 \in L_0$ is central in $G$, we see that $\varphi(g)=wgw^{-1}=v g v^{-1}$ for all $g \in H$, thus $\varphi$ is indeed an inner automorphism of $H$.
\end{proof}

An automorphism $\alpha \in Aut(G)$ is said to be {\it pointwise inner} if $\alpha(g)$ is conjugate to $g$ for each $g \in G$.
The set of all pointwise inner automorphisms, $Aut_{pi}(G)$ is a normal subgroup of $Aut(G)$.

In \cite{Grossman} Grossman proved that if $G$ is a finitely generated conjugacy separable group such that $Aut_{pi}(G)=Inn(G)$ then $Out(G)$ is residually finite. We thus arrive at the following

\begin{cor} \label{cor:OutGResFinite}
	Let $G$ be a finitely presented residually free group. Then $Out(G)$ is residually finite.
\end{cor}

\end{document}